\documentclass[11pt]{article}
\usepackage{amsmath}
\usepackage{amsthm}
\usepackage{amssymb}
\usepackage{dsfont}

\newtheorem{theo}{Theorem}[section]

\newtheorem{prop}[theo]{Proposition}
\newtheorem{lemma}[theo]{Lemma}

\newtheorem{claim}[theo]{Claim}

\newcommand{\eps}{{\varepsilon}}
\newcommand{\PP}{{\mbox{\rm Prob}}}
\newcommand{\var}{ Var  }
\begin{document}
\date{}

\title{
Random necklaces require fewer cuts
%\\Draft, {\small randneck17.tex}
}

\author{Noga Alon
\thanks
{Department of Mathematics, Princeton University, Princeton,
NJ 08544, USA and Schools of Mathematics and Computer Science,
Tel Aviv University, Tel Aviv, Israel.
Email: {\tt nalon@math.princeton.edu}.
Research supported in part by
NSF grant DMS-1855464 and
BSF grant 2018267.
}
\and Dor Elboim
\thanks{Department of Mathematics, Princeton University,
Princeton, NJ 08544, USA. Email:
	{\tt delboim@princeton.edu}.}
\and J\'anos Pach
\thanks{R\'enyi Institute, Budapest and MIPT, Moscow.
Supported by NKFIH grant K-131529, Austrian Science Fund Z 342-N31,
Ministry of Education and Science of the Russian Federation
MegaGrant No. 075-15-2019-1926, ERC Advanced Grant ``GeoScape.'' Email:
{\tt pach@cims.nyu.edu}.}
\and  G\'abor Tardos
\thanks{R\'enyi Institute, Budapest and MIPT, Moscow.
Supported by the ERC Synergy Grant ``Dynasnet'' No. 810115, the ERC Advanced Grant ``GeoScape'', the National Research, Development and Innovation Office (NKFIH) grants K-132696, SNN 135643 and by the Russian Federation
MegaGrant No. 075-15-2019-1926. Email:
{\tt tardos@renyi.hu}.}
}

\maketitle
\begin{abstract}
It is known that any open necklace with beads of $t$ types
in which the number of beads of each type is divisible by $k$, can be
partitioned by at most $(k-1)t$ cuts into intervals that can be
distributed into $k$ collections, each containing the same number
of beads of each type. This is tight for all values of
$k$ and $t$.

Here, we consider the case of random
necklaces, where the number of beads of each type is $km$. Then
the minimum number of cuts required for a ``fair'' partition with the above
property is a random variable $X(k,t,m)$. We prove that for fixed $k,t,$ and
large $m$, this random variable is at least $(k-1)(t+1)/2$ with
high probability. For $k=2$, fixed $t$, and large $m$, we determine the
asymptotic behavior of the probability that $X(2,t,m)=s$ for all values
of $s\le t $. We show that this probability is polynomially small when
$s<(t+1)/2$, it is bounded away from zero when $s>(t+1)/2$, and decays like
$\Theta ( 1/\log m)$ when $s=(t+1)/2$.

We also show that for large $t$,
$X(2,t,1)$ is at most $(0.4+o(1))t$ with high
probability and that for large $t$ and large ratio $k/\log t$,
$X(k,t,1)$ is
$o(kt)$ with high probability.
%The problem of determining the asymptotic behaviour of $X(2,t,1)$ for large $t$ remains open.
\vspace{0.2cm}

\noindent
{\bf AMS Subject classification:} 05D40, 60C05\\    {\bf Keywords:} necklace, random walk, second moment method
\end{abstract}

\section{Introduction}

Let $N$ be an \emph{open} necklace
with the same number of beads belonging to each of $t$ classes,
and suppose that this number is divisible by $k$.
As was proved in \cite{Al}, it is always possible to cut $N$ in at most
$(k-1)t$ points and distribute the resulting intervals
into $k$ collections, each containing the same number
of beads of each type. This is tight for all values of the
parameters, as shown by any necklace in which the beads
of each type appear contiguously. A possible interpretation of the
result is the following. Suppose that $k$ mathematically oriented thieves
want to fairly distribute a necklace among
each other. If the necklace is comprised of the same number of beads belonging to $t$
different types (colors) and this number is a multiple of $k$, then they
can always achieve this by making at most
$(k-1)t$ cuts between beads.
\smallskip

The aim of the present paper is to study this problem for
random necklaces. The random model we consider here is a
necklace of total length $n=ktm$ consisting of exactly $km$ beads
of type $i$ for each $1 \leq i \leq t$, chosen uniformly among all
intervals of $n$ beads as above.  Call a set of cuts
of such a necklace {\em fair}, if it is possible to split the resulting
intervals into $k$ collections, each containing exactly $m$ beads
of each type. For a necklace $N$, let
$X=X(N)$ be the minimum number of cuts in a fair collection. When
$N$ is chosen randomly as above, $X$ is a random variable which
we denote by $X(k,t,m)$. By the result of \cite{Al} mentioned
above, we have $X(k,t,m) \leq (k-1)t$ with probability $1$. Our
objective is to study the typical behavior of the random
variable $X=X(k,t,m)$. All results presented here are asymptotic, where at
least one of the three variables $k,t,m$ tends to infinity. As
usual, we say that
a result holds {\em with high probability} ({\em whp}, for
short), if the probability that it holds tends to $1$, when the relevant
parameter(s) tend to infinity.

The problem of determining the asymptotic behavior of $X(k,t,m)$
turns out to be connected to several seemingly unrelated topics,
including matchings in nearly regular hypergraphs and random walks
in Euclidean spaces. For some values of the parameters
$k,m$ and $t$,
we were able to determine
this behavior quite accurately,
showing that
the typical number of required cuts is sometimes significantly smaller
than the deterministic upper bound $(k-1)t$ which is always valid.

We start with the following observation.
\begin{prop}
\label{p11}
For every fixed $k$ and $t$, as $m$ tends to infinity, we have
$X=X(k,t,m) \geq\left\lceil \frac{(k-1)(t+1)}{2}\right\rceil$ whp.
\end{prop}

In our main result, we describe the asymptotic behavior of
$X=X(k,t,m)$ for two thieves ($k=2$) and any fixed number of types
$t$, as $m$ tends to infinity.

\begin{theo}\label{thm:big m}
	Let $t$ be a fixed positive integer and $m\to \infty $.
	\begin{enumerate}
		\item
		For all $1\le s < \frac{t+1}{2}$, we have that
		\begin{equation}
		\mathbb P \big( X(2,t,m)=s \big) =
                \Theta \big( m ^{s-\frac{t+1}{2}} \big).
		\end{equation}
		\item
		When $t$ is odd and $s=\frac{t+1}{2}$, we have
		\begin{equation}
		\mathbb P \big( X(2,t,m)=s \big) =
                \Theta \Big( \frac{1}{\log m} \Big).
		\end{equation}
		\item
		For all  $\frac{t+1}{2} <s\le t $, we have that
		\begin{equation}
		\mathbb P \big( X(2,t,m)=s \big)  =\Theta (1).
		\end{equation}
	\end{enumerate}
\end{theo}

The most surprising aspect of the last result and, in fact, of the present paper is that the distribution of the values of
$ X(2,t,m)$ is not concentrated: they are spread over the interval $(\frac{t+1}2, t]$, each value being assumed with positive probability.

We also consider the case $m=1$, in which every thief should get a
single bead of each type.
\begin{theo}
\label{t14}
For $t$ and $k/ \log t$ tending to infinity,
the random variable $X=X(k,t,1)$ is
$o(kt)$ whp.
\end{theo}

For $k=2$, $m=1$ and large $t$, we prove
\begin{theo}
\label{t15}
The random variable
$X=X(2,t,1)$ is at least $2H^{-1}(1/2)t-o(t)=0.220...t -o(t)$ whp, where
$H^{-1}(x)$ is the inverse of the binary entropy
function $H(x)=-x \log_2 x -(1-x) \log_2 (1-x)$ taking values
in the interval $[0,1/2]$.

On the other hand, $X\le 0.4 t+o(t)$ holds whp.
\end{theo}

The upper bound above was obtained jointly with Ryan Alweiss,
Colin Defant and Noah Kravitz.
We show that both the upper and the lower bounds can be
slightly improved, but the arguments leave a considerable gap
between the two estimates. We further prove that
the probability that $X(2,t,1)$ deviates from its expectation by
at least $C \sqrt t$ is at most $O(e^{-\Omega(C^2)})$.
\smallskip

The rest of this paper is organized as follows. In the next section,
we analyze the case where $k$ and $t$ are fixed. We start with the
simple proof of Proposition~\ref{p11}. Since the proof of
Theorem~\ref{thm:big m} is rather long and technical,
in Subsection~\ref{ss2.1}, we settle the special case $t=3$,
which requires similar ideas but is considerably simpler.
The argument, in its full generality,
encompasses the following three subsections.

In Section~\ref{section3}, we study the case where each thief
gets one bead of each type and the number of types, $t$, tends
to infinity and present the proofs of Theorems~\ref{t14}
and~\ref{t15}. We conclude our paper with
several remarks and open problems.

To simplify the presentation, throughout the
paper, we ignore all floor and ceiling signs, whenever they are not
crucial.

\smallskip

\section{Fixed $k$ and $t$: Proof of Theorem~\ref{thm:big m}}

Let $N$ be a random open necklace with $km$ beads in each of the $t$
types. Let $P$ be a {\em partition of the positions}
of the beads into $k$ parts.
%(it is not a random variable).
We call
$P$ \emph{balanced} if each part contains $tm$ beads and we call it
\emph{fair} if each part contains exactly $m$ beads of each type. Thus,
only balanced partitions can be fair. Let $k=2$ and consider two balanced
partitions $P_1$ and $P_2$. For a part $H_1$ of $P_1$, we can choose a
part $H_2$ of $P_2$ with $q:=|H_1\setminus H_2|\le tm/2$. We call $q$
the \emph{distance} of the balanced partitions $P_1$ and $P_2$. (Note
that the distance is symmetric and does not depend on the choice of the
part $H_1$ of $P_1$.)

We will use the first and second moment methods. Therefore, we need to
estimate the probability that a given partition is fair and that two given
partitions are both fair. The asymptotic notation in this section refers
to behavior as $m$ goes to infinity, while parameters $k$, $t,$ and $s$
(where appropriate) remain fixed.

\begin{claim}\label{12}
\begin{itemize}
\item[(i)] The probability $P(k,t,m)$ that a fixed balanced partition is fair is
$$(1+o(1))\frac{t^{(k-1)/2}k^{(t-1)/2}}{(2\pi m)^{(k-1)(t-1)/2}}=\Theta_{k,t}(m^{-(k-1)(t-1)/2}).$$
\item[(ii)] If $k=2$ and the distinct balanced partitions $P_1$ and $P_2$ have distance $q$, then we have
$$\mathbb P(P_1, P_2\mbox{ are fair})=\Theta_t((qm)^{-(t-1)/2}).$$
\end{itemize}
\end{claim}

\noindent{\bf Proof:}\, First, we show part (i). The total number of necklaces is
$(ktm)!/((mk)!)^t$. For any fixed balanced partition, the number of necklaces making it fair is $((mt)! )^k/ (m!)^{tk}$. Therefore, we have
$$
P(k,t,m)=\frac{((mt)! )^k/ (m!)^{tk}}{(ktm)!/((mk)!)^t}=(1+o(1))\frac{t^{(k-1)/2}k^{(t-1)/2}}{(2\pi m)^{(k-1)(t-1)/2}}
$$
The claimed estimate follows by Stirling's formula.
\smallskip

For (ii), choose parts $H_1$ of $P_1$ and $H_2$ of $P_2$ with
$|H_1\setminus H_2|=q$. Assuming $P_1$ is fair, then $P_2$ is also fair
if and only if in each color, $H_1\setminus H_2$ and $H_2\setminus H_1$
contain the same number of beads. Note that the distributions of the
colors in these two intervals are independent so-called {\em multivariate
hypergeometric distributions}. In this case, we choose a uniform random
$q$-set from a base set containing $m$ beads from each of the $t$
types. The conditional probability $\mathbb P(H_2\mbox{ is fair}\mid H_1\mbox{
is fair})$ is, therefore, the same as the probability that two independent
variables from the same multivariate hypergeometric distributions
agree. This is easily seen to be $\Theta_t(q^{(1-t)/2})$. For
completeness, we sketch the proof of this statement below.

The random sample can be characterized by the numbers
$q_i$ of beads of type $i$, for every $i$. We must
have $q_i\ge0$, $\sum_{i=1}^tq_i=q$. If this holds, then the probability of this specific outcome is
$$P_{q_1,\dots,q_t}=\frac{(mt-q)!q!/
\prod_{i=1}^t((m-q_i)!q_i!)}{(mt)!/(m!)^t}.$$
To prove the upper bound $O_t(q^{(1-t)/2})$ for the probability
that two independent samples agree, it is enough to notice that
$P_{q_1,\dots,q_t}=O_t(q^{(1-t)/2})$ holds for any sequence $(q_1,\dots,q_t)$.

For the lower bound, note that each $q_i$ has a hypergeometric
distribution with mean $q/t$ and variance $O_t(q)$, so with a
constant probability the numbers $q_i$ will simultaneously satisfy
$|q_i-q/t|=O_t(\sqrt q)$ for every $i$. There are $O_t(q^{(t-1)/2})$ such
integer vectors also satisfying $\sum_{i=1}^tq_i=q$, so the collision
probability is $\Omega_t(q^{(1-t)/2})$. This proves the bound and the
claim.\hfill$\Box$
\medskip

The proof of Proposition~\ref{p11} is a simple first
moment argument based on part (i) of the previous claim.
\medskip

\noindent
{\bf Proof of Proposition \ref{p11}:}\, Let $N$ be an open necklace
consisting of $km$ beads of each of the $t$ types.

We estimate the number $n(s,k,m,t)$ of balanced partitions that can be
obtained by $s$ cuts of the open necklace. Note that $s$ cuts result in
$s+1$ intervals, and the partition can be reconstructed from the ordered
list of the lengths of theses intervals together with the information which part of
the partition contains which interval. In fact, we can save by not
specifying the length of the last interval in each part, because it can
be computed from the lengths of the other intervals in the same part, as
the partition must be balanced. The intervals have lengths at most $tm$,
so we have $$n(s,k,m,t)\le k^{s+1}(tm)^{s+1-k}=O_{s,k,t}(m^{s-k+1}).$$

From Claim~\ref{12}(i),
the expected number of fair partitions obtained by $s$ cuts is
$$
n(s,k,m,t) P(k,t,m)=O_{k,s,t}( m^{s-k+1-(t-1)(k-1)/2})
=O_{k,s,t}( m^{s-(t+1)(k-1)/2}).
$$
For any fixed $k$ and $t$, if $s< (t+1)(k-1)/2$, then the above
expectation tends to $0$ as $m$ tends to infinity. This implies the
assertion of Proposition~\ref{p11}. \hfill $\Box$
\medskip

\subsection{Proof of Theorem~\ref{thm:big m} for $t=3$}\label{ss2.1}

The proof of Theorem~\ref{thm:big m} in its full generality is rather lengthy and technical.
In  the present subsection, we settle the special case $t=3$. For this special case, the argument is much simpler. However, it is based on some of the same ideas as the general proof.

One cut is sufficient to fairly distribute the random necklace $N$
(that is, $X(N)=1$) if and only if the partition of $N$ into its
first and second halves is fair, which has probability $\Theta(1/m)$
by Claim~\ref{12}(i). According to the (deterministic)
result of Alon~\cite{Al}, for every $N$, we have $X(N) \leq 3$. Thus, it remains to show that
the probability that $X(N)\le2$ is $\Theta(1/\log m)$.  In order to
estimate this probability, note that two cuts suffice if and only if
there is a balanced partition of $N$ into two cyclic intervals that is
fair. There are exactly $3m$ balanced partitions into two cyclic
intervals. For $0\le i<3m$, we denote by $P_i$ the balanced partition
into an interval starting at position $i+1$ and ending at position $i+3m$,
and its complement.

Let $Y=Y(N)$ denote the random variable counting the number of fair
partitions into cyclic intervals. Clearly, $X(N) \leq 2$ if and only if $Y$
is positive.  We first establish a lower bound for the probability of this
event using the second moment method. The random variable $Y$ is a sum
of $3m$ random variables $Y_i$, where $Y_i$ is the indicator variable for
$P_i$ being fair. Each of these indicator random variables has expectation
$P(2,3,m)=(1+o(1))\frac{\sqrt3}{\pi m}$ by Claim~\ref{12}(i), so
$$
\mathbb E(Y)= \sum_{i=0}^{3m-1}\mathbb E(Y_i)=3m \cdot (1+o(1)) \frac{\sqrt 3}{\pi m}=
(1+o(1)) \frac{3\sqrt 3}{\pi}=\Theta(1).
$$
The expected value of $Y^2$  is
$$
\sum_{0 \leq i, j \leq 3m-1}
\mathbb E(Y_iY_j),
$$
where the sum is taken over ordered pairs. For $i=j$, we have
$\mathbb E(Y_i^2)=\mathbb E(Y_i)=P(2,3,m)=\Theta(1/m)$ by Claim~\ref{12}(i). If $i\ne
j$ we have $\mathbb E(Y_iY_j)=\Theta(1/(mq))$ by Claim~\ref{12}(ii), where
$q$ is the distance between the corresponding partitions. Note that
$q=\min(|i-j|,3m-|i-j|)$. Therefore, for all $1\le q<3m/2$, we have $6m$
pairs $(i,j)$ with $q$ being the distance between $P_i$ and $P_j$. If
$m$ is even, we also have the special case $q=3m/2$ with half as many
terms. We can ignore this special case, as we are only interested in
the order of magnitude. Therefore,
$$
\mathbb E(Y^2) =3m\cdot\Theta\left(\frac1m\right)+ 6m \sum_{1 \leq q \leq 3m/2}
\Theta\left(\frac{1}{mq}\right)=\Theta (\log m).
$$
By the Paley-Zygmund Inequality \cite{PZ}, \cite{PZ1} or, equivalently,
by the Chung-Erd\H{o}s Inequality \cite{CE},
$$
\mathbb P(Y>0) \geq \frac{\mathbb E(Y)^2}{\mathbb E(Y^2)}=\Theta\left(\frac{1}{\log
m}\right).
$$

We next prove an upper bound for the probability that $Y$ is positive.
To this end, we define another random variable $Z=Z(N)$. We will show
that $Z$ is positive with probability $O(1/\log m)$, and the probability
that $Y$ is positive but $Z$ is not, is even lower. The crucial step in
bounding the probability that $Z$ is positive, is the analysis of the
probability that an appropriate two-dimensional random walk does not
return to the origin in a certain number of steps. For this, we apply
a slightly modified version of a classical argument of Dvoretzky and
Erd\H{o}s (\cite{DE}, see also \cite{ET}, \cite{Re}). This is the subject
of Claim~\ref{p2} below.

Let $Z=Z(N)$ denote the number of fair partitions into two cyclic
intervals such that if we shift the parts by at most, say,
$s=\left\lceil \sqrt m\,\right\rceil$ positions to the right, then the
resulting partition is no longer fair. Note, first, that if $Y$ is positive
and $Z$ is zero, then in every set of $s$ consecutive balanced partitions
into two cyclic intervals, there is at least one fair partition. However,
in this case, we have $Y\ge3m/s>\sqrt m$ and, as the expectation of  $Y$
is $O(1)$, the probability of this event is $O(1/\sqrt m)$.

Next, we need to bound the probability that $Z$ is positive. For this,
we use a first moment (union bound) argument. The variable $Z$ is the
sum of $3m$ indicator variables and, by symmetry, these variables have
the same expected value. Therefore, we have
\begin{eqnarray*}
\mathbb P(Z>0)&\le&\mathbb E (Z)\\
&=&3m\cdot \mathbb P(P_0\mbox{ is fair, but no $P_i$ is fair for }1\le i\le s)\\
&=&3m\PP_1\PP_2,
\end{eqnarray*}
where $\PP_1=\mathbb P(P_0\mbox{ is fair})$ and $\PP_2=\mathbb P(\mbox{no $P_i$ is fair for }1\le i\le s\mid P_0\mbox{ is fair})$. We have $\PP_1=\Theta(1/m)$, by Claim~\ref{12}(i),\ and $\PP_2=O(1/\log m)$, by Claim~\ref{p2} below. All this yields
$$\mathbb P(Z>0)=3m\cdot\Theta(1/m)\cdot O(1/\log m)=O(1/\log m).$$

Combining this bound with our earlier estimate for the probability that $Y>0$ and $Z=0$, implies that
$$\mathbb P(Y>0)\le O(1/\log m)+O(1/\sqrt m)=O(1/\log m).$$
This completes the proof of the theorem
for $t=3$,
modulo Claim~\ref{p2}, which we still have to establish.\hfill$\Box$
\medskip

\begin{claim}\label{p2}
The conditional probability $\PP_2$ defined above satisfies $\PP_2=O(1/\log m)$.
\end{claim}

\noindent{\bf Proof:}\, To evaluate $\PP_2$, we will always assume that
$P_0$ is fair and consider $N$ to be a random necklace satisfying this
condition. For simplicity, we call the three types of beads blue, red, and green, respectively.
For $1\le i\le s$, let $b_i$ be the (signed) difference
between the number of blue beads in positions $1$ through $i$ and
the number of blue beads in positions $3m+1$ through $3m+i$. Let
$r_i$ be similarly calculated for the red beads. Clearly, $P_i$
is fair if and only if $b_i=r_i=0$. We consider the vectors $(b_i,r_i)$,
as locations of a random walk starting at $(b_0,r_0)=(0,0)$. The steps
$(b_i,r_i)-(b_{i-1},r_{i-1})$ can be calculated as the difference between
two vectors from the set $S=\{(0,0), (1,0),(0,1)\}$, one corresponding
to the color of bead at position $i+3m$, the other to the color of the
bead at position $i$.

For $1\le i\le s$, let $P(i)$ denote the probability that this walk
returns to the origin after $i$ steps, that is, $b_i=r_i=0$ (equivalently,
$P_i$ is fair), and let $Q(i)$ denote the probability that the random
walk does not return to the origin for $i$ steps. By Claim~\ref{12}(ii),
we have $P(i)=\Theta(1/i)$. We also have $\PP_2=Q(s)$.

It is simpler to analyze the random walk if it is memoryless, that is,
if the steps are independent random variables. This is not exactly the
case, but is ``almost true.'' Let the random variables $b'_i$ and $r_i'$
be calculated in the same way after we fill all the positions from
$1$ to $s$ and from $3m+1$ to $3m+s$ with beads of independently and
uniformly distributed random colors. In this case, we have a memoryless
random walk that starts at the origin, and each step can be obtained
as the difference of two vectors from $S$ selected uniformly at random.
A fixed arrangement of $b$ blue, $r$ red, and $g$ green beads in
the first $s$ positions ($b+r+g=s$) has probability exactly $1/3^s$
in the second model, while its probability in the first model is
$$\frac{(3m-s)!/((m-b)!(m-r)!(m-g)!)}{(3m)!/(m!)^3}=\left(\frac13+O\left(\frac sm\right)\right)^s=\Theta\left(\frac1{3^m}\right).$$
A similar statement is true for the colors in positions $3m+1$ through $3m+s$, and in both models, the color arrangements in the two intervals are independent. Therefore, the probabilities of the same event in the two models differ by at most a constant factor.

For $1\le i\le s$, we define $P'(i)$ as the probability that this memoryless random walk returns to the origin after $i$ steps and $Q'(i)$ as the probability that the modified random walk does not return for $i$ steps. By the above argument, we have $P'(i)=\Theta(P_i)=\Theta(1/i)$ and $\PP_2=\Theta(Q'(s))$.

In the following calculation, we split the possible walks according
to their last visit at the origin and use $P'(0)=Q'(0)=1$.
$$1 = \sum_{i=0}^s P'(i) Q'(s-i)\ge Q'(s)\sum_{i=0}^sP'(i),$$
where the inequality comes from the fact that $Q'(i)$ is a monotone
decreasing function of $i$. We have
$$\sum_{i=0}^s P'(i) =1+\sum_{i=1}^s\Theta\left(
\frac{1}{i}\right) = \Theta( \log s)=\Theta(\log m).
$$
Thus, $\PP_2=\Theta(Q'(s))=O(1/\log m)$, as needed.\hfill$\Box$
\medskip

\subsection{The second moment method}

In this subsection, we prove part (1) of Theorem~\ref{thm:big m} as well as the lower bound in part (2) of Theorem~\ref{thm:big m}.

Let $N$ be a random open necklace with $2m$ beads of each of $t$ types.
Fix an integer $s \leq t$ for the number of cuts.
Recall that in the proof of Proposition~\ref{p11}, we calculated the number of balanced partitions achievable with $s$ cuts as $n(s,2,t,m)=O_{s,t}(m^{s-1})$. However, here we can be more precise. We call a balanced partition achievable by $s$ cuts but not with fewer cuts an $s$-cut partition. In an $s$-cut partition the necklace is cut into $s+1$ non-empty segments and these segments alternate
between the two participants. Taking the odd numbered intervals we obtain an arbitrary partition of $tm$ beads to $\lceil(s+1)/2\rceil$ nonempty intervals while the even numbered intervals partition $tm$ beads to $\lfloor(s+1)/2\rfloor$ nonempty intervals. Therefore, the exact number of $s$-cut partitions is
$$
n'(s,t,m)= {{tm -1} \choose\lceil(s-1)/2\rceil} \cdot {{tm -1} \choose\lfloor(s-1)/2\rfloor}=\Theta_{s,t}(m^{s-1}).
$$

By Claim~\ref{12}(i), the probability that a fixed balanced partition
is fair is $P(2,t,m)=\Theta_t(m^{-(t-1)/2})$.

Let $Y$ be the random variable counting the
number of fair $s$-cut partitions. By linearity of
expectation and the estimates above we have
\begin{equation}\label{eq:expectation of Y}
\mathbb E (Y)=n'(s,t,m) P(2,t,m)=\Theta_{s,t}(m^{s-(t+1)/2}).
\end{equation}

It follows from Markov's inequality that
\begin{equation}\label{eq:upper bound with s cuts}
\mathbb P \big( X(2,t,m)=s \big) \le \mathbb P (Y>0 ) \le \mathbb E(Y) = O_{s,t} \big( m^{s-\frac{t+1}{2}} \big) .
\end{equation}
This finishes the proof of the upper bound in part (1) of Theorem~\ref{thm:big m}.

We note that for odd $t$ and $s=(t+1)/2$, the above expectation is
$\Theta_t(1)$ and therefore the upper bound in part (2) of
Theorem~\ref{thm:big m} does not follow from
Markov's inequality as in part (1). We will present the considerably more involved proof of that upper bound in Subsection~\ref{random walks}.

We proceed to estimate the second moment of $Y$, the number of fair $s$-cut partitions.  For an $s$-cut partition $P$ let $Y_P$ denote the
indicator random variable whose value is $1$ if and only if $P$ is fair.
By Claim~\ref{12}(ii), for any two distinct balanced
partitions $P$ and $P'$, the probability that both $P$ and $P'$ are fair is
$$
\mathbb E(Y_PY_{P'})=\Theta_t((mq)^{-(t-1)/2}),
$$
where $q$ is the distance between $P$ and $P'$. For $0<q \le tm/2$,
let $n^*(q,t,m,s)$ denote the number of ordered pairs of $s$-cut
partitions $(P,P')$, such that the distance between $P$ and $P'$ is $q$.
To estimate the number $n^*(q,t,m,s)$,
consider the collection of $2s$ cuts of both partitions, where if
both contain a cut at the same point we take it twice. These cuts
partition the interval of beads into $2s+1$ pairwise disjoint intervals
(including possibly some empty intervals, when the two partitions
share a cut). Let us select a part $H$ of $P$ and a part $H'$ of $P'$ such that $|H\setminus H'|=q$. The non-empty intervals can be classified as follows: the ones belonging to $H \cap H'$, the ones belonging to $H \setminus H'$, to
$H' \setminus H$, and those not in $H \cup H'$. The total number of beads in
the intervals of the first type is $|H\cap H'|=tm-q$, and this is also the
total number of beads in the intervals of the fourth type. The
number of beads in intervals of the second type is $|H\setminus H'|=q$, and so is
the number of beads in intervals of the third type. Call the first
and fourth types {\em even} and the second and third {\em odd} (this indicates the
parity of the number of sets among $H,H'$ to which the
corresponding interval belongs). With appropriate classification of the empty intervals into one of the four types one can make sure that even and odd intervals alternate so we either have $s+1$ even and $s$ odd intervals or vice versa.

We can reconstruct both partitions $P$ and $P'$ from the ordered list of
types and lengths of all the $2s+1$ intervals. In fact, we can save by
not specifying the length of the last interval in each type as that can be
computed from the lengths of the other intervals. We clearly have at most
$4^{2s+1}$ possibilities for the sequence of types. Even intervals have
lengths between $0$ and $tm-q$ and odd intervals have lengths between $0$
and $q$. So, in total, we have at most $4^{2s+1}(tm-q+1)^{s-1}(q+1)^{s-2}$
possibilities if there are $s+1$ odd and $s$ even intervals, and at most
$4^{2s+1}(tm-q+1)^{s-2}(q+1)^{s-1}$ in the reversed case. We have
$$n^*(q,t,m,s)=O_{s,t}(m^{s-1}q^{s-2}),$$
because the estimate holds for both of
these numbers if $s$ and $t$ are fixed.

Recall that the random variable $Y$ can be written as $Y=\sum Y_P$, where
$P$ ranges over the $s$-cut partitions.
Therefore
\begin{equation}\label{eq:second moment of Y}
\begin{split}
\mathbb E(Y^2)&\leq\mathbb E(Y) + \sum_{1\le q\leq tm/2} n^*(q,t,m,s)\Theta((mq)^{-(t-1)/2})\\
&= \Theta_{s,t} (m^{s-(t+1)/2})+\sum_{1\le q\leq tm/2}
O_{s,t} (m^{s-(t+1)/2}
q^{s-1-(t+1)/2}).
\end{split}
\end{equation}

When $s<(t+1)/2$, the last inequality shows that
$\mathbb E(Y^2)=O_{s,t} \big( m ^{s-\frac{t+1}{2}}  \big)$. Thus, using \eqref{eq:expectation of Y} and  Paley-Zygmund Inequality, we get
\begin{equation*}
\mathbb P \big(X(2,t,m) \le s \big)\ge\mathbb P (Y>0) \ge \frac{\mathbb E(Y )^2 }{\mathbb \mathbb E(Y^2)} = \Omega _{s,t}  \big(  m^{s-\frac{t+1}{2}} \big) .
\end{equation*}
Combining this result with \eqref{eq:upper bound with s cuts}, we obtain for all $1\le s< \frac{t+1}{2}$,
\begin{equation*}
\mathbb P \big( X(2,t,m)=s \big)=\mathbb P \big( X(2,t,m)\le s \big)-\mathbb P \big( X(2,t,m)\le s-1 \big) = \Omega _{s,t}  \big( m^{s-\frac{t+1}{2}} \big) .
\end{equation*}
This finishes the proof of part (1) of Theorem~\ref{thm:big m}.

When $t$ is odd and $s=\frac{t+1}{2}$, the inequality \eqref{eq:second moment of Y} shows that $\mathbb E (Y^2)=O_{s,t}\big( \log m  \big)$. Thus, using the same arguments we get that
\begin{equation*}
\mathbb P \big( X(2,t,m)=s \big) = \Omega _{s,t} \Big(  \frac{1}{\log m} \Big) .
\end{equation*}
This finishes the proof of the lower bound in part (2) of Theorem~\ref{thm:big m}.

Finally, we note that when $t$ is even and $s=t/2+1$, we get from \eqref{eq:second moment of Y} that $\mathbb E (Y^2 ) =O_{s,t}(\sqrt{m})$ and therefore by the same arguments
\begin{equation}\label{eq:bounded away for s=t/2+1}
\mathbb P \big( X(2,t,m)=s \big) =\Omega _{t} (1).
\end{equation}

\noindent
This is a special case of part (3) of Theorem~\ref{thm:big m}. The proof in the general case is similar, but requires an additional twist. We present it the next subsection. Since the computation involved is similar to the one above, we omit some of the details.
\vspace{0.2cm}

\subsection{Proof of part (3) of Theorem~\ref{thm:big m}}

We have $\mathbb P(X(2,t,m)=s)=0$ for $s>t$,
by the deterministic result, and
$\mathbb P(X(2,t,m)=s)$ goes to zero if $s\le (t+1)/2$ by parts (1) and (2) of Theorem~\ref{thm:big m}. We proved part (1) in the previous subsection and will prove the relevant direction of part (2) in the next. Our goal here is to prove
$\mathbb P(X(2,t,m)=s)=\Omega_t(1)$ in all remaining cases $t/2+1\le s\le t$. According to
\eqref{eq:bounded away for s=t/2+1},
this is true for $s=t/2+1$ and can be proved
by straightforward second moment argument. For the general case we will also use the second moment method, but with a modified distribution.

We will use the following simple claim. It holds for real intervals too, but here we use the word ``interval'' to represent any set of consecutive elements in a sequence (such as beads on a necklace).
\begin{claim}
\label{chyper}
There exists an absolute positive constant $c$ so that the
following holds.
Let $x$, $y$ and $U$ be positive and let
$X$ be a uniform random subset of $x$ points of an interval
$Y$ of length $y$.

Then the probability that there exists an interval $Z$
(of any length) in $Y$ so that  $|X \cap Z|$
deviates from its expectation,
$|Z|x/y$, by at least
$U$ is  at most $8e^{-cU^2/x}$.
\end{claim}

\noindent{\bf Proof:}
If there exists a interval of $Y$ with deviation above $U$, then there also exists an initial interval (starting at the left end of $Y$) with deviation larger than $U/2$, so it suffices to consider initial intervals. It is also enough to considers initial intervals of length at most $y/2$ by symmetry. Consider the
first $y/2$ elements of $Y$
one by one from left to right, exposing for each of them
in turn if it belongs to the random set $X$. If during the process we ever
reach an initial interval in which the number of elements of $X$
deviates from its expectation by more
than $U/2$, stop the process
and reveal all remaining elements of $X$.
Conditioning on having a large deviation where we stopped the process,
with probability at least, say, $1/4$ we still have deviation
of at least $U/4$ in the interval
of the $y/2$ first points of $Y$. But this probability is  at most
$e^{cU^2/x}$ for some absolute positive constant $c$, by
standard estimates for large deviations of a hypergeometric
distribution (see \cite{Hoe} or \cite{JLR}, Theorem 2.10 and Theorem 2.1.)
This implies that the probability of any initial segment of $Y$ of length at most $y/2$ has deviation above $U/2$ is at most $4e^{-cU^2/x}$. The probability of an initial segment of any length existing with such a high deviation is at most twice this and if no such initial interval exists then the deviation of any interval is at most $U$.
This proves the claim. \hfill$\Box$

Returning to the proof of the theorem, recall that $s$, $t$ and $m$ are
positive integers satisfying $t/2+1\le s\le t$. We will treat $t$ and $s$
as a constants and assume in our calculation that $m$ is sufficiently
large depending on $t$. Let $D$ be the uniform distribution over necklaces
$N$ with $2m$ beads of type $i$ for every $1 \leq i \leq t$. We need to
prove that $\mathbb P_D(X(N)=s)=\Omega_t(1)$. In what follows we consider
another distribution $D'$ on some of the same possible necklaces obtained
in a two step process as follows.

We split
the necklace $N$ into $s$ intervals $I_i$, each consisting
of  $2mt/s$ consecutive beads. Further we split each interval $I_i$ into three equal length subintervals $I_{i,1}$, $I_{i,2}$ and $I_{i,3}$ out which $I_{i,2}$ lies in the middle. Strictly speaking, some rounding is necessary unless $3s$ divides $2mt$, but we ignore these roundings as they do not matter in our calculations. We will choose the positive integer $L<m/2$ later.  In the first step of our two step process we place $L$ random beads of type $i$ uniformly in both of the intervals $I_{i,1}$ and $I_{i,3}$, for every $i$. We call the beads so placed \emph{seeds}, so we have $2sL$ seeds in total. As the second step of our process generating the distribution $D'$ we place the remaining $2mt-2sL$ beads (the \emph{non-seeds}, $2m-2L$ of them of type $i$ for $i\le s$ and $2m$ beads of type $i$ for $i>s$) uniformly in the available slots.

For a type $i$ and an interval $J$ in the necklace we denote the number of beads of type $i$ in $J$ by $n_i(J)$. The dependence on $N$ is implicit. We call a necklace $N$ \emph{normal} if the distribution of types in every interval is close to its expectation, that is, if for every interval $J$ and every type $i$, we have $|n_i(J)-\mathbb E_{D'}(n_i(J))|<L/(4t)$.

We use the asymptotic notations $O_t(\cdot)$, $\Omega_t(\cdot)$ and $\Theta_t(\cdot)$ to hide positive multiplicative factors depending on $t$ alone. These factors are not allowed to depend on $m$ or $L$. (Dependence on $s$ is allowed as $s\le t$ can take finitely many distinct values for a fixed $t$.)

\begin{claim}\label{normal}
$$\mathbb P_{D'}(N\hbox{ is not normal})=O_t(1)e^{-\Omega_t(L^2/m)}$$
\end{claim}

\noindent{\bf Proof:}\, We can identify $2s+t$ sets placed uniformly in
the process defining the distribution $D'$. For $i\le s$ we have two
sets of seeds of type $i$ placed uniformly in the intervals $I_{i,1}$
and $I_{i,3}$, respectively. And for any type $i$ we have the set of
non-seed beads of type $i$ placed uniformly in the positions not occupied
by seeds. We apply Claim~\ref{chyper} for each of these processes. The
union bound yields the estimate stated in the lemma for the existence of
an interval in any of these processes where the number of beads placed in
the interval deviates from its expectation by more than $L/(12t)$. It
remains to prove that assuming no deviation exceeds $L/(12t)$ the
resulting necklace $N$ is normal.

To see this, let us fix an interval $J$ and a type $i$. Let $E$ stand
for the expectation of $n_i(J)$ in the distribution $D'$ and let $E'$
stand for the same expectation conditioned on the placement of the seeds
(so $E'$ is a random variable as $n_i(J)$ but $E$ is a constant). Note
that $E'$ is determined by how many of the seeds are placed inside $J$
in each of the $2s$ relevant subintervals. This is deterministic for
all but at most two of the subintervals (the ones containing the ends of
$J$). Note also that the dependence of $E'$ on these numbers is linear
with all coefficients below $1$ in absolute value. With our low deviation
assumption on the seeds this means that $|E'-E|<2L/(12t)$ as $E$ is the
expectation of $E'$. We also assumed that the number of non-seeds of
type $i$ in $J$ differs from its expectation after the seeds are placed
by at most $L/(12t)$. So we have $|n_i(J)-E'|\le L/(12t)$ and therefore
$|n_i(J)-E|<L/(4t)$ as needed.\hfill $\Box$

\begin{claim}\label{atleast}
If $N$ is normal, then $X(N)\ge s$.
\end{claim}

\noindent{\bf Proof:}\, We prove the contrapositive: If $X(N)<s$, then $N$ is not normal. So let us fix a fair partition of $N$ with fewer than $s$ cuts. Note that one of the intervals $I_i$ is not cut at all. Fix such an $i$ and note that one of the players receive no part of $I_i$. We look at the expected value of type $i$ beads (necessarily non-seeds) in the intervals he receives. For any $j$, a position of $I_{j,2}$ receives a bead of type $i$ with probability $(2m-2L)/(2mt-2sL)\le1/t$, while a position in $I_{j,1}$ or $I_{j,3}$ receives a seed with probability $3sL/(2mt)$, so it receives a non-seed of type $i$ with probability $(1-3sL/(2mt))(2m-2L)/(2mt-2sL)\le1/t-3sL/(2mt^2)$. As the partition is balanced the player receives $mt$ beads. There are only $2mt/3$ beads in the middle subintervals $I_{j,2}$, so at least $mt/3$ beads are coming from non-middle subintervals and therefore the expected number of type $i$ beads this part contains is at most $m-sL/(4t)$. Our partition is fair, so the actual number of type $i$ beads the player receives is exactly $m$. The discrepancy is coming from the at most $s$ intervals the player receives, so the actual number of type $i$ beads in one of those deviates from its expectation by at least $L/(4t)$. This proves that $N$ is not normal.\hfill $\Box$

\begin{claim}\label{atmost}
$$\mathbb P_{D'}(X(N)\le s)=\Omega_t(1)$$
\end{claim}

\noindent{\bf Proof:}\, We call a balanced partition of the necklace between the two players \emph{central} if it is obtained from $s$ cuts, one in each of the middle intervals $I_{i,2}$ by distributing the resulting $s+1$ intervals alternately between the two players. Note that the seeds are distributed equally between the players in a central partition: each players receive exactly $L$ seeds of each type. Thus, a central partition is fair if and only if both players receive an equal number on non-seeds of type $i$ for each $i$. The distribution of the non-seeds are uniform on the available slots, so both parts of Claim~\ref{12} applies: the probability under the distribution $D'$ that a central partition is fair is $\Theta_t(m^{-(t-1)/2})$ and the probability of two distinct central partitions are simultaneously fair is $\Theta_t((mq)^{-(t-1)/2})$, where $q$ is the distance between the central partitions. It should be acknowledged that the situation here differs from the situation considered in Claim~\ref{12} in that there we have $2m$ beads of each type whereas here we have $2m-2L$ non-seed beads in types $1\le i\le s$. But as $2m-2L$ is between $m$ and $2m$ the estimate still holds and the original proof of Claim~\ref{12} applies in this modified setting almost verbatim. Note that the hidden constant in our estimate does not depend on $L$.

We apply the Paley-Zygmund Inequality for the random variable $Y$ counting the fair central partitions. This argument is very similar to the one presented at the end of the previous subsection. We have $\Theta(m^{s-1})$ central partitions, so
$$\mathbb E_{D'}(Y)=\Theta_t(m^{s-1})\Theta_t(m^{-(t-1)/2})=\Theta_t(m^{s-(t+1)/2}).$$
We calculate $\mathbb E_{D'}(Y^2)$ as the sum of probabilities for ordered pairs of central partitions that they are simultaneously fair. We have $\Theta_t(s^{m-1})$ pairs of equal partitions and for $1\le q$ we have $O(m^{s-1}q^{s-2})$ pairs of distance $q$. This gives
$$\mathbb E_{D'}(Y^2)=O_t(m^{s-(t+1)/2})+\sum_{q=1}^{mt/2}O(m^{s-(t+1)/2}q^{s-(t+1)/2-1})=O(m^{2s-t-1}),$$
where we used $s>(t+1)/2$ in the last step. Using the Paley-Zygmund Inequality we obtain
$$\mathbb P_{D'}(Y>0)\ge\frac{(\mathbb E_{D'}(Y))^2}{\mathbb E_{D'}(Y^2)}=\Omega_t(1).$$

To finish the proof of the claim simply observe that $Y>0$ means that there is a fair central partition, so $X(N)\le s$ as central partitions have $s$ cuts.\hfill $\Box$
\medskip

We set $L$ such that $N$ is normal with probability at least $1-\mathbb P_{D'}(X(N)\le s)/2$. By Claims~\ref{normal} and \ref{atmost} this can be achieved by an appropriate choice also satisfying $L=O(\sqrt m)$. With this choice of $L$ we have that $N$ is normal and $X(N)\le s$ with probability $\Omega_t(1)$. In this case we actually have $X(N)=s$ by Claim~\ref{atleast}. These estimates hold for a random necklace $N$ in the distribution $D'$. To prove part (3) of Theorem~\ref{thm:big m} we need a similar estimate in the uniform distribution $D$. The following Claim finishes the proof because it establishes that if $L=O(\sqrt m)$, then $D'$-weight of a normal necklace is only constant times its $D$-weight.

\begin{claim}
We have $P'(N)/P(N)\le e^{O_t(L^2/m)}$ for any normal necklace $N$, where $P'(N)$ stands for the probability of obtaining $N$ in the distribution $D'$ and $P(N)$ is the probability of obtaining it in $D$.
\end{claim}

\noindent{\bf Proof:}\, This calculation is tedious but very elementary. The probability $P'(N)$ depends on $N$ through the number $n_i(I_{i,j})$ of beads of type $i$ in the interval $I_{i,j}$ for $1\le i\le s$ and $j=1$ or $3$. We have $\binom{n_i(I_{i,j})}L$ choices to select seeds in the interval $I_{i,j}$ consistent with $N$. Thus, the probability of selecting all seeds consistent with $N$ is
$$\frac{\prod_{i=1}^s\binom{n_i(I_{i,1})}L\binom{n_i(I_{i,3})}L}{\binom{2mt/(3s)}L^{2s}}.$$
The distribution of non-seeds is uniform, so after such a consistent choice of the seeds we obtain $N$ with probability
$$\frac{(2m-2L)!^s(2m)!^{t-s}}{(2mt-2sL)!}.$$
For a random necklace in the distribution $D'$, the interval $I_{i,j}$ ($j=1$ or 3) contains $L$ seeds of type $i$, and the expected number of non-seeds of type $i$ is $(2m-2L)(2mt/(3s)-L)/(2mt-2sL)\le2m/(3s)-L/t$. As $N$ is normal, the actual value $n_i(I_{i,j})$ deviates from its expectation by less than $L/t$, so we have $n_i(I_{i,j})<2m/(3s)+L$. Using this estimate and the calculations above, we obtain
$$P'(N)\le\frac{\binom{2m/(3s)+L}L^{2s}((2m-2L)!^s(2m)!^{t-s}}{\binom{2mt/(3s)}L^{2s}(2mt-2sL)!}.$$
As $D$ is uniform, $P(N)$ does not depend on $N$:
$$P(N)=\frac{(2m)!^t}{(2mt)!}.$$
To estimate $P'(N)/P(N)$, we use the inequalities $(a-b)^b<\binom abb!\le a^b$ and obtain
$$\frac{P'(N)}{P(N)}\le\left(\frac{2mt(2m/(3s)+L)}{(2m-2L)(2mt/(3s)-L)}\right)^{2sL}.$$
Using $s\le t$ and $L\le m/2$, we can further estimate
$$\frac{2mt(2m/(3s)+L)}{(2m-2L)(2mt/(3s)-L)}\le 1+O_t\left(\frac Lm\right),$$
so we have
$$\frac{P'(N)}{P(N)}\le\left(1+O_t\left(\frac Lm\right)\right)^{2sL}=e^{O_t(L^2/m)},$$
as claimed\hfill $\Box$

\subsection{Random walks}\label{random walks}

In this subsection, we prove the upper bound in part (2) of
Theorem~\ref{thm:big m}. The following proposition is at the heart
of the argument. In this proposition we bound the probability that a
certain equation cannot be solved in the trajectories of independent
random walks. To prove the proposition we modify and generalize a
proof of Lawler~\cite{Lawler1}. In \cite{Lawler1} and \cite{Lawler2}
Lawler studied the probability that the traces of two independent random
walks on $\mathbb Z ^4$ are disjoint. See also \cite{erdos taylor2,ET}
for works of Erd\H{o}s and Taylor on the same problem.

Throughout this subsection, we fix $t\ge 1$ odd, $s=\frac{t+1}{2}$ and let
the $O$ notations depend on $t$. Recall that an infinite two-sided random
walk $W(n)$ is a sequence of random variables in some Euclidean space such
that $W(n+1)-W(n)$ for $n\in \mathbb Z $ are independent and identically
distributed. We say that the walk has a finite range if there is a
finite set $A$ for which $\mathbb P ( W(1)-W(0)\notin A )=0$. Finally,
the walk $W$ is called centered if $\mathbb E [W(1)-W(0)]=0$.

\begin{prop}\label{prop:Lawler}
	Let $\langle W_j(n) ,\ n\in \mathbb Z  \rangle$ for $j\le s$ be independent and identically distributed two-sided random walks on $\mathbb Z^{2s-2}$. Suppose that $W_j(0)=0$ for all $j$ and that the walks are centered and have a finite range. Then, there exists $C>0$ depending on $s$ and the step distribution of the walks such that for all $N\ge 2$,
	\begin{equation*}
	\mathbb P \bigg( \forall k=(k_1,\dots ,k_{s})\in A_N, \ \sum_{j=1}^{s} W_j(k_j) \neq 0 \bigg) \le \frac{C}{\log N},
	\end{equation*}
	where
	\begin{equation*}
	A_N:=\bigg\{ k\in \mathbb Z ^{s} : \sum _{j=1}^{s} k_j =0 \text{ and } k>_{\ell } 0 \bigg\} \cap [-N,N]^{s},
	\end{equation*}
	and where $>_{\ell } $ is the lexicographic order on $\mathbb Z ^{s}$.
\end{prop}

We first show how to use Proposition~\ref{prop:Lawler} in order to prove the upper bound in part (2) of Theorem~\ref{thm:big m}.

\subsubsection{Proof of the upper bound in part (2) of Theorem~\ref{thm:big m}}
We parameterize a partition of the necklace with $s$ cuts by a vector of integers $i=(i_0, i_1,\dots ,i_s, i_{s+1})$ with
\begin{equation*}
0=i_0\le i_1\le \cdots \le i_s\le i_{s+1}=2mt.
\end{equation*}
This corresponds to a partition where the first thief gets beads $1$ to $i_1$, the second one gets beads $i_1+1$ to $i_2$ and so on.
Let $I$ be the set of balanced partitions. That is
\begin{equation*}
I:=\bigg\{ (i_0,i_1,\dots ,i_{s+1}) \ \Big| \  \sum _{j=0}^{s} (-1)^j \left( i_{j+1}-i_j \right) =0  \bigg\}.
\end{equation*}
For $n\le 2mt$, let $U(n) \in \mathbb N ^{t-1}$ be the random variable that, in the $j$'th coordinate, counts the number of beads of type $j$ out of the first $n$ beads. It is clear that $i=(i_0, i_1,\dots i_{s+1})\in I$ is fair if and only if
\begin{equation}\label{eq:U}
\sum _{j=0}^s (-1)^j (U(i_{j+1})-U(i_j))=0.
\end{equation}
Indeed, equation \eqref{eq:U} says that the first $t-1$ types are equally distributed between the thieves. Therefore, as each thief gets in total $mt$ beads, it follows that the last type must be equally distributed as well.

Next, let $Z$ be the number of fair partitions in $I$. It suffices to prove that
\begin{equation}\label{eq:what we need}
\mathbb P (Z>0) = O \bigg( \frac{1}{\log m} \bigg).
\end{equation}
To this end, define the sets of partitions
\begin{equation*}
I_1:=\left\{ i \in I \ \big| \  \forall 0\le j\le s, \ i_{j+1}-i_j > 2m^{\frac{1}{4}}  \right\}, \quad I_2:=I \setminus I_1,
\end{equation*}
and let $Z_2$ be the number of fair partitions in $I_2$.

We define a total order on $I$. For $i,i'\in I$, we write $i'\succ i$ if
\begin{equation*}
(i_1',-i_2',i_3',-i_4',\dots ) >_{\ell } (i_1,-i_2, i_3,-i_4,\dots ),
\end{equation*}
where $<_{\ell }$ is the lexicographic order on $\mathbb Z ^s$. For a partition $i\in I_1$, define the set
\begin{equation*}
B_i: =\left\{ i'\in I \ \big| \ i'\succ i \text{ and } \forall j \le s, \  |i_j-i_j'|\le m^{\frac{1}{4}} \right\}
\end{equation*}
and the event
\begin{equation*}
\mathcal B _i := \left\{ i \text{ is fair and } \forall i'\in B_i, \ i' \text{ is not fair } \right\}.
\end{equation*}
Finally, let
\begin{equation*}
Z_1:=\sum _{i\in I_1} \mathds 1 _{\mathcal B _i}.
\end{equation*}

We claim that
\begin{equation}\label{eq:Z_1 Z_2}
\{Z>0\} \subseteq \{ Z_1>0 \} \cup \{Z_2>0\}.
\end{equation}

Indeed, suppose that $Z>0$ and let $i\in I$ be the maximal fair partition
with respect to $\succ $. If $i \in I_2$ then $Z_2>0$ and therefore we
may assume that $i \in I_1$. Since $i$ is maximal, for all $i'\in B_i$,
$i'$ is not fair. Thus, $\mathcal B _i $ holds and $Z_1>0$.

We turn to bound the probabilities of the two events on the right-hand side of \eqref{eq:Z_1 Z_2}. It is easy to check that $|I_2| =O (m^{s-\frac{7}{4}})$  and therefore by Markov's inequality and Claim~\ref{12}(i),
\begin{equation}\label{eq:bound on Z_2}
\mathbb P (Z_2>0) \le \mathbb E (Z_2) =\sum _{i \in I_2 } \mathbb P (i \text{ is fair}) \le |I_2| \cdot O \big( m^{-\frac{t-1}{2}} \big) = O\big(  m^{-\frac{3}{4}} \big).
\end{equation}

Next, we bound $\mathbb P ( Z_1>0 )$. We have that
\begin{equation}\label{eq:conditional}
\mathbb P (\mathcal B _i) = O \big(  m^{-\frac{t-1}{2}} \big)   \cdot \mathbb P \left(  \forall i'\in B_i, \ i' \text{ is not fair } | \ i \text{ is fair} \right).
\end{equation}
We bound the last probability in the following claim.

\begin{claim}\label{claim:coupling}
	We have that
	\begin{equation}\label{eq:conditional probability}
	\mathbb P \left(  \forall i'\in B_i, \ i' \text{ is not fair } | \ i \text{ is fair} \right)= O \bigg( \frac{1}{\log m} \bigg).
	\end{equation}
\end{claim}

Using Claim~\ref{claim:coupling}, equation \eqref{eq:conditional} and Markov's inequality we get
\begin{equation}\label{eq:bound on Z_1}
\mathbb P (Z_1>0) \le \mathbb E (Z_1) =\sum _{i\in I_1} \mathbb P (\mathcal B _i) = O \bigg( \frac{m^{-\frac{t-1}{2}}}{\log m} \bigg) \cdot  |I_1|= O \bigg( \frac{1}{\log m} \bigg).
\end{equation}
Finally, \eqref{eq:what we need}  follows from \eqref{eq:bound on Z_1}, \eqref{eq:bound on Z_2} and \eqref{eq:Z_1 Z_2}. This finishes the proof of part (2) of Theorem~\ref{thm:big m}.
\medskip

It remains to prove Claim~\ref{claim:coupling}

\begin{proof}[\bf Proof of Claim~\ref{claim:coupling}:]
	
	Throughout this proof, we consider a uniform necklace $N$ with $2m$ beads of each type such that the partition $i$ is fair. As before, we denote by $U(n)\in \mathbb N ^{t-1}$ the counting vector. In this case, by \eqref{eq:U}, $i'\in I$ is fair if and only if
	\begin{equation}\label{eq:i, i'}
	\sum _{j=1}^{s} (-1)^{j+1} \left( U(i'_j)-U(i_j) \right) =0.
	\end{equation}
	
	For $1 \le j \le s$ and $|n|\le m^{\frac{1}{4}}$, define
	\begin{equation*}
	\tilde{W}_j(n):=t(-1)^{j+1} \left(  U(i_j+(-1)^{j+1}n)- U(i_j) \right) -(n,\dots ,n).
	\end{equation*}
	Using this notation and \eqref{eq:i, i'} we have that $i'\in B_i$ is fair if and only if
	\begin{equation*}
	\sum _{j=1}^s \tilde{W}_j\left( (-1)^{j+1}(i_j' -i_j) \right)=0.
	\end{equation*}
	Here, we also used that
	\begin{equation*}
	\sum _{j=1}^s (-1)^{j+1}(i_j'-i_j)=0,
	\end{equation*}
	which follows as $i,i'\in I$. Thus, letting
	\begin{equation*}
	A:= \left\{ \left( (-1)^{j+1} (i_j'-i_j) \right)_{j=1}^s \ | \ i'\in B_i  \right\} \subseteq \mathbb Z ^s,
	\end{equation*}
	we obtain that the probability in \eqref{eq:conditional probability} is given by
	\begin{equation*}
	\mathbb P \bigg(\forall k\in A, \ \sum _{j=1}^s \tilde{W}_j(k_j) \neq 0 \bigg).
	\end{equation*}
	Moreover, from the definition of $B_i$ it is clear that
	\begin{equation*}
	A=\bigg\{ k\in \mathbb Z ^s : \sum _{j=1}^{s} k_j =0 \text{ and } k>_{\ell } 0 \bigg\} \cap [-m^{\frac{1}{4}},m^{\frac{1}{4}}]^s.
	\end{equation*}
	
	We cannot use Proposition~\ref{prop:Lawler} yet, because the processes $\tilde{W}_j$ are not exactly independent random walks. We
will show that it is possible to couple them with random walks. To this end, let $Y$ be a random variable in $\mathbb Z ^{t-1}$ with distribution
	\begin{equation*}
	\forall j \le t-1 , \ \mathbb P (Y=t e_j-(1,\dots ,1))=\frac{1}{t} \quad \text{and} \quad \mathbb P (Y=-(1,\dots ,1))=\frac{1}{t}.
	\end{equation*}
	Let $W_j(n)$ for $j \le s$ and $|n|\le m^{\frac{1}{4}}$ be independent two-sided random walks on $\mathbb Z ^{t-1}$ with steps distributed like $Y$ and with $W_j(0)=0$. We claim that one can couple $\{W_j \}_{j=1}^s$ and $\{ \tilde{W}_j \}_{j=1}^s$ such that
	\begin{equation}\label{eq:couple W}
	\mathbb P \left( \forall j, \ \tilde{W}_j=W_j \right) = 1-O\big( m^{-\frac{1}{2}} \big) .
	\end{equation}
	
We give a sketch of proof for this fact. Consider a new necklace
	$N'$ of length $l = 2t\lfloor m^{\frac{1}{4}} \rfloor$ obtained from $N$ by concatenating
	intervals of beads of length $2 \lfloor m^{\frac{1}{4}} \rfloor$ around each of
	the cuts. We also let $N''$ be another necklace of the same length
	obtained by choosing independently and uniformly the type of each
	bead. It is not hard to see that \eqref{eq:couple W} follows
	from a coupling of $N'$ and $N''$ such that $\mathbb P (N'=
	N'')= 1-O (m^{-\frac{1}{2}})$. The first bead of $N'$ is clearly
	uniform. Conditioning on the first $m$ beads of $N'$, a simple
	counting argument shows that the probability that the next bead
	is of a certain type is $\frac{1}{t} +O(m^{-\frac{3}{4}})$. Thus,
	the probability to get any particular sequence of beads of length
	$l$ is
	\begin{equation*}
	\Big( \frac{1}{t} +O(m^{-\frac{3}{4}}) \Big) ^l =\frac{1}{t^l} \left( 1 +O(m^{-\frac{1}{2}}) \right) .
	\end{equation*}
	This shows that $N'$ and $N''$ can be coupled such that $\mathbb P (N'\neq N'') = O ( m^{-\frac{1}{2}} )$, which proves~\eqref{eq:couple W}.
	
	By \eqref{eq:couple W} and Proposition~\ref{prop:Lawler}, we obtain that
	\begin{equation*}
	\mathbb P \bigg(\forall k\in A, \ \sum _{j=1}^s \tilde{W}_j(k_j) \neq 0 \bigg) = O \big( m^{-\frac{1}{2}} \big)  +\mathbb P \bigg(\forall k\in A, \ \sum _{j=1}^s W_j(k_j) \neq 0 \bigg)= O \bigg( \frac{1}{\log m} \bigg).
	\end{equation*}
	This finishes the proof of Claim~\ref{claim:coupling}.
	
\end{proof}

\subsubsection{Proof of Proposition~\ref{prop:Lawler}}

The proof is somewhat similar to the proof of Claim~\ref{p2}. In the proof of Claim~\ref{p2}, we split the possible walks according to their last visit at the origin. Then, we use the Markov property to argue that the probability that the last visit is at time $k$ equals the probability that the walk returns to the origin at time $k$ times the probability that the walk avoids the origin in the remaining time.

Similarly, in the following proof, we consider the last integer vector $k$ (according to the lexicographic order) for which $\sum W_j(k_j)=0$. Then, we condition on the trajectories of all walks other then the first one and use the Markov property for the first walk. The proof becomes slightly more technical as we need to control probabilities conditioned on the last walks with high probability. Throughout the proof we let the $O$ notations depend on the step distribution of the random walk as well as on $s$.

Consider the set
\begin{equation*}
\mathcal E :=\bigg\{ \xi =(\xi _2,\dots ,\xi _{s}) \ \bigg| \  \xi _j:[-2N,2N]\cap \mathbb Z \to \mathbb Z^{2s-2}  \bigg\}.
\end{equation*}
We think of $\mathcal E $ as the set of trajectories of the walks $W_2,\dots ,W_s$. Let $K=(K_1,\dots ,K_{s} )$ be the maximal element $(k_1,\dots ,k_{s})\in A_N \cup \{0\} $ with respect to $>_{\ell } $ such that
\begin{equation*}
\sum _{j=1}^{s} W_j(k_j)=0.
\end{equation*}

For $\xi \in \mathcal E$ and $k\in A_N$ define the event
\begin{equation*}
\mathcal A (\xi ,k):=\{K=k\}\cap \bigcap _{j=2}^{s} \{\forall l\in \{ -2N,\dots , 2N\},  \  W_j(k_j+l)-W_j(k_j)=\xi _j(l) \}.
\end{equation*}
It is clear that these events are disjoint and that
\begin{equation}\label{eq:supset}
\begin{split}
\mathcal A (\xi ,k) &\supseteq  \bigg\{ W_1(k_1)=\sum _{j=2}^{s} \xi _j(-k_j) \bigg\} \\
&\cap \bigcap _{j=2}^{s} \Big\{ \forall l\in \{ -2N,\dots ,2N\}, \  W_j(k_j+l)-W_j(k_j)=\xi _j(l) \Big\} \\
&\cap \bigg\{ \forall n \in A_{2N} ,\ W_1(k_1+n_1)-W_1(k_1)+\sum _{j=2}^{s } \xi _j(n_j)\neq 0  \bigg\}.
\end{split}
\end{equation}
All the s+1 events whose intersection gives the right-hand side of \eqref{eq:supset} are independent (the first and last events are independent as $k_1,n_1\ge 0$) and, therefore,
\begin{equation}\label{eq:lower bound on probability}
\begin{split}
\mathbb P \left( \mathcal A (\xi ,k) \right) \ge P(\xi ) &\cdot  \mathbb P \bigg(  W_1(k_1)=\sum _{j=2}^{s} \xi _j(-k_j) \bigg)\\
&\cdot \mathbb P \bigg(  \forall n \in A_{2N} ,\ W_1(n_1)+\sum _{j=2}^{s } \xi _j(n_j)\neq 0 \bigg),
\end{split}
\end{equation}
where
\begin{equation*}
P(\xi ):= \prod _{j=2}^{s} \mathbb P \Big( \forall l\in \{-2N,\dots ,2N\}, \  W_j(l)=\xi _j(l) \Big).
\end{equation*}

Next, for $\xi \in \mathcal E$ define the function
\begin{equation*}
G(\xi ):= \sum _{k \in A_N} \mathbb P \bigg(  W_1(k_1)=\sum _{j=2}^{s} \xi _j(-k_j) \bigg).
\end{equation*}

The following lemma shows that $G$ is typically of the order of $\Omega ( \log N )$.

\begin{lemma}\label{lem:upper bound on probability}
	There exists $c >0$ such that
	\begin{equation*}
	   \mathbb P \big( G(W_2,\dots ,W_s) <c \log N  \big) = O \bigg( \frac{1}{\log N} \bigg).
	\end{equation*}
\end{lemma}

We postpone the proof of Lemma~\ref{lem:upper bound on probability} and proceed with the proof of Proposition~\ref{prop:Lawler}. To this end, define the set
\begin{equation*}
\mathcal G   :=\Big\{ \xi \in \mathcal E   : G(\xi ) \ge c \log N \Big\}.
\end{equation*}
Summing the inequality in \eqref{eq:lower bound on probability} over $\xi \in \mathcal G   $ and $k \in A_N$, we get
\begin{equation*}
\begin{split}
1\ge \mathbb P &\bigg( \bigcup _{\xi \in \mathcal G   } \bigcup _{k \in A_N} \mathcal A (\xi ,k) \bigg)=\sum _{\xi \in \mathcal G   } \sum _{k \in A_N} \mathbb P \left( \mathcal A (\xi ,k) \right) \\
&\ge \sum _{\xi \in \mathcal G   } P(\xi ) \cdot G(\xi ) \cdot  \mathbb P \bigg(  \forall n \in A_{2N} ,\ W_1(n_1)+\sum _{j=2}^{s } \xi _j(n_j)\neq 0 \bigg) \\
&=  \Omega ( \log N ) \sum _{\xi \in \mathcal G   } P(\xi )  \cdot  \mathbb P \bigg(  \forall n \in A_{2N} ,\ W_1(n_1)+\sum _{j=2}^{s } \xi _j(n_j)\neq 0 \bigg) \\
&=\Omega ( \log N ) \cdot   \mathbb P \bigg( (W_2,\dots , W_{s}) \in \mathcal G   \ \text{ and } \ \forall n \in A_{2N} , \  \sum _{j=1}^{s } W_j (n_j)\neq 0 \bigg)\\
&= \Omega ( \log N ) \cdot   \mathbb P \bigg(   \ \forall n \in A_{2N} , \  \sum _{j=1}^{s } W_j (n_j)\neq 0 \bigg) -O(1).
\end{split}
\end{equation*}
where the last equality follows from Lemma~\ref{lem:upper bound on probability}. Thus,
\begin{equation}
\mathbb P \bigg(   \ \forall n \in A_{2N} , \  \sum _{j=1}^{s } W_j (n_j)\neq 0 \bigg) = O \bigg( \frac{1}{\log N} \bigg).
\end{equation}
This completes the proof of Proposition~\ref{prop:Lawler}.

We turn to prove Lemma~\ref{lem:upper bound on probability}. For the
proof we need the following standard claim on integer valued random walks.

\begin{claim}\label{claim:on random walk}
	Let $W(n)$ be a centered, finite range random walk on $\mathbb Z ^d$ with $W(0)=0$. Suppose that the covariance matrix of $W(1)$ is $\Sigma $ and that $\Sigma $ is nonsingular. Then
	\begin{enumerate}
		\item
		For all $n\ge 1$,  $k \in \mathbb Z ^d$ and $a>0$ such that $||k|| \le a\sqrt{n}$ and $\mathbb P (W(n)=k)>0$, we have that
		\begin{equation*}
		\mathbb P (W(n)=k) = \Omega  _a \big( n^{-\frac{d}{2}} \big).
		\end{equation*}
		\item
		For all $n\ge 1 $ and $r>0$, we have that
		\begin{equation*}
		\left| \mathbb P (|| W(n) ||\le r\sqrt{n} ) - \mathbb P (||Z||\le r  )  \right| = O \Big(  \frac{1}{\sqrt{n}} \Big),
		\end{equation*}
		where $Z\sim N(0,\Sigma )$.
	\end{enumerate}
\end{claim}

\begin{proof}[\bf Proof:]
	The second part follows from \cite[Theorem~1.1]{Berry esseen}.
	
	Now we prove the first part. We start by showing that without loss of generality, $W$ is irreducible. Indeed, otherwise consider the set
	\begin{equation*}
	\Lambda := \{ k\in \mathbb Z ^d :  \exists n \ge 0, \ \mathbb P (W(n)=k)>0 \}.
	\end{equation*}
	For a general walk, $\Lambda $ is a semigroup. However, since $W$ is centered, it is not hard to check that $\Lambda $ is in fact a lattice. If we let $T$ be a linear transformation that maps $\Lambda $ to $\mathbb Z ^d$ then the new walk $T(W)$ is an irreducible walk. Now, we can use \cite[Theorem~3.1]{LLT}  that determines the asymptotic behavior of $\mathbb P (W(n)=k)$ for an irreducible (possibly periodic) walk. See also the work of Polya on the simple random walk \cite{polya}.
\end{proof}

\begin{proof}[\bf Proof of Lemma~\ref{lem:upper bound on probability}:]
	Fix $k_2,\dots ,k_{s}$ such that $-N/s \le k_2,\dots ,k_s < 0$ and let $k_1:=-k_2-\cdots -k_{s}$. It is clear that $k=(k_1,\dots ,k_{s})\in A_N$. On the event
	\begin{equation*}
	\mathcal B := \Big\{ \forall 2\le j \le s , \ ||W_j( -k_j )||\le \sqrt{|k_j|}  \Big\},
	\end{equation*}
	we have that
	\begin{equation*}
	\bigg| \bigg| \sum _{j=2}^{s} W_j(-k_j) \bigg| \bigg| \le  \sqrt{|k_2|}+\cdots +\sqrt{|k_{s}|} \le s \sqrt{k_1}
	\end{equation*}
	and therefore, by the first part of Claim~\ref{claim:on random walk}, on $\mathcal B$ we have
	\begin{equation*}
	\mathbb P \bigg( W_1(k_1) = \sum _{j=2}^{s} W_j(-k_j) \ \bigg| \ W_2,\dots W_{s}\bigg) = \Omega \Big(  \frac{1}{k_1^{s-1}} \Big).
	\end{equation*}
	Note that in order to use Claim~\ref{claim:on random walk} we have to first verify that the last probability is positive. This is indeed the case since, with positive probability, the walk $W_1$ takes the same $-k_2$ steps as $W_2$, then the same $-k_3$ steps as $W_3$ and so on. We obtain that
	\begin{equation}\label{eq:lower bound on G}
	\begin{split}
	    G(W_2,\dots ,W_{s} ) &\ge \Omega(1) \sum _{k_2=1}^{N/s} \cdots \sum _{ k_{s} =1}^{N/s} \frac{1}{\big( k_2+\cdots +k_s\big)^{s-1}} \prod _{j=2}^{s}\mathds 1 \big\{ || W_j(k_j) ||\le \sqrt{k_j} \big\}\\
	    &\ge \Omega (1) \sum _{p=1} ^{\log N} 2^{-p(s-1)} \sum _{k_2=2^p}^{2^{p+1}-1} \cdots  \sum _{k_s=2^p}^{2^{p+1}-1} \prod _{j=2}^{s}\mathds 1 \big\{ || W_j(k_j) ||\le \sqrt{k_j} \big\}.
	\end{split}
	\end{equation}
	Denote by $X$ the sum on the right-hand side of \eqref{eq:lower bound on G}. Intuitively, $X$ is of order $\log N$ since each dyadic scale $p$ contributes order $1$ to the sum and far away scales are weakly correlated. We make this heuristic rigorous using the second moment method.
	
	By the central limit theorem we have that
	\begin{equation*}
	\mathbb E (X) =  \sum _{p=1} ^{\log N} 2^{-p(s-1)} \sum _{k_2=2^p}^{2^{p+1}-1} \cdots  \sum _{k_s=2^p}^{2^{p+1}-1} \Omega(1)  = \Omega ( \log N ).
	\end{equation*}
	We turn to bound the variance of $X$. We have that
	\begin{equation*}
	\begin{split}
	&\var (X ) \le \sum _{p=1}^{\log N} \sum _{q=1}^{\log N}  2^{-(p+q)(s-1)} \sum _{k_2=2^p}^{2^{p+1}}  \cdots \sum _{k_{s}=2^p}^{2^{p+1}} \sum _{l_{2}=2^q}^{2^{q+1}} \cdots  \sum _{l_{s}=2^q}^{2^{q+1}}   \prod _{j=2} ^{s} \\
	&\Big(  \mathbb P \left(||W_j(k_j)||\le \sqrt{k_j}, \ ||W_j(l_j)||\le \sqrt{l_j} \ \right) -\mathbb P \left( ||W_j(k_j)||\le \sqrt{k_j} \right)\mathbb P \left( ||W_j(l_j)||\le \sqrt{l_j} \right) \Big).
	\end{split}
	\end{equation*}
	For all $j$ with $k_j\le l_j$, we have that
	\begin{equation}\label{eq:corelation of walk}
	\begin{split}
	&\mathbb P \left(||W_j(k_j)||\le \sqrt{k_j}, \ ||W_j(l_j)||\le \sqrt{l_j} \ \right) \\
	& \quad \quad \quad \quad \quad \quad  \quad  \le \mathbb P \left(||W_j(k_j)||\le \sqrt{k_j}, \ ||W_j(l_j)-W_j(k_j)||\le \sqrt{l_j}+\sqrt{k_j} \ \right)\\
	& \quad \quad \quad \quad \quad \quad \quad = \mathbb P \left(||W_j(k_j)||\le \sqrt{k_j} \ \right) \mathbb P \left(  ||W_j(l_j-k_j) ||\le \sqrt{l_j}+\sqrt{k_j} \right).
	\end{split}
	\end{equation}
	When $l_j \ge 2 k_j$, let
	\begin{equation*}
	r:= \frac{\sqrt{l_j}+\sqrt{k_j}} {\sqrt{l_j-k_j} }= 1+O\bigg(\frac{\sqrt{k_j}}{\sqrt{l_j}}\bigg).
	\end{equation*}
	By the second part of Claim~\ref{claim:on random walk} for $l_j \ge 2 k_j$,
	\begin{equation}\label{eq:use the Berry Esseen}
	\begin{split}
	\mathbb P \left(  ||W_j(l_j-k_j) ||\le \sqrt{l_j}+\sqrt{k_j} \right)=\mathbb P &\left( ||Z||\le r \right) +O\bigg(\frac{1}{\sqrt{l_j}}\bigg)\\
	=\mathbb P (||Z||\le 1) +O\bigg(\frac{\sqrt{k_j}}{\sqrt{l_j}}\bigg) &=\mathbb P \left(  ||W_j(l_j) ||\le \sqrt{l_j} \right) +O\bigg(\frac{\sqrt{k_j}}{\sqrt{l_j}}\bigg).
	\end{split}
	\end{equation}
	It is clear that the probability in the left-hand side of \eqref{eq:use the Berry Esseen} is estimated by the right-hand side of \eqref{eq:use the Berry Esseen} when $k_j \le l_j\le 2k_j$ and therefore it holds whenever $l_j\ge k_j$. Substituting this estimate into \eqref{eq:corelation of walk} and using the same arguments when $l_j \le k_j$, we get that for all $k_j,l_j$,
	\begin{equation*}
	\begin{split}
	\mathbb P \left(||W_j(k_j)||\le \sqrt{k_j}, \ ||W_j(l_j)||\le \sqrt{l_j} \ \right) & \\
	\le \mathbb P \left(||W_j(k_j)||\le \sqrt{k_j} \right) &\mathbb P \left(  ||W_j(l_j)||\le \sqrt{l_j} \ \right) + O \bigg( \frac{\sqrt{\min (l_j,k_j)}}{\sqrt{\max (l_j,k_j)}} \bigg) .
	\end{split}
	\end{equation*}
	Thus,
	\begin{equation}\label{eq:variance annoying sum}
	\begin{split}
	\var (X ) &\le O(1) \sum _{p=1}^{\log N} \sum _{q=p}^{\log N} 2^{-(p+q)(s-1)} \sum _{k_2=2^p}^{2^{p+1}}  \cdots \sum _{k_{s}=2^p}^{2^{p+1}} \sum _{l_{2}=2^q}^{2^{q+1}} \cdots  \sum _{l_{s}=2^q}^{2^{q+1}} 2^{(p-q)(s-1)/2} \\
	&\le O(1)\sum _{p=1}^{\log N} 2^{p(s-1)/2} \sum _{q=p}^{\log N}  2^{-q(s-1)/2} \le O(1)\sum _{p=1}^{\log N} 1 =O(\log N).
	\end{split}
	\end{equation}
	Finally, by Chebyshev's inequality there exists $c>0$ such that
	\begin{equation*}
	\mathbb P (X<c \log N) = O \bigg(  \frac{1}{\log N} \bigg) .
	\end{equation*}
	This finishes the proof of the lemma using \eqref{eq:lower bound on G}.
	\end{proof}

\section{The case $m=1$, many types}\label{section3}

In this section, we prove Theorems \ref{t14} and \ref{t15}
in which the number of beads of each type is equal to the number of
thieves and each collection of intervals should contain exactly one
bead of each type. In order to prove Theorem \ref{t14}, we need a
hypergraph edge-coloring result of Pippenger and Spencer \cite{PS}.

First, we recall the terminology. A \emph{hypergraph} is a pair $H=(V,E)$, where $V$ is the set of \emph{vertices} and $E$ is a multiset consisting of subsets of vertices called the \emph{edges}. We only consider here finite hypergraphs. We say that $H$ is \emph{$C$-uniform} if every edge contains $C$ vertices. The \emph{degree} $d_H(v)$ of a vertex $v\in V$ is the number edges containing $v$. If all the vertices have the same degree $k$, we call $H$ \emph{$k$-regular}. The \emph{codegree} of two distinct vertices is the number of edges which contain both of them. The \emph {maximum degree} and \emph{maximum codegree} of the hypergraph $H$ is the maximum degree of a vertex of $H$, and the maximum codegree of two distinct vertices in $H$, respectively. A set of pairwise disjoint edges is called a \emph{matching}.

\begin{theo}[\cite{PS}]
\label{t41}
For every integer $C \geq 2$  and every $\eps>0$, there is
$\delta>0$ such that the following statement holds.

For every positive integer $k$ and any $C$-uniform hypergraph $H$ with maximum degree at most $k$ and maximum codegree at most $\delta k$, the edges of $H$ can be partitioned into at most $(1+\eps)k$ matchings.
\end{theo}

The original statement published in \cite{PS} had two extra assumptions. First, it was required that the number of vertices of the hypergraph is sufficiently large as a function of $C$ and $\eps$. Second, it was also assumed that every vertex has degree at least $(1-\delta)k$. However, these two conditions are superfluous. One can add any number of isolated vertices to a hypergraph $H$ to satisfy the former condition. Then one can use the following lemma to obtain a $k$-regular hypergraph $H'$ containing $H$ without changing the maximal codegree. Partitioning the edge set of $H'$ into few matchings also partitions the edge set of $H$ into the same number of matchings.

\begin{lemma}
\label{regular}
For any $C\ge2$ and any $C$-uniform hypergraph $H$ with at least one edge and maximum degree at most $k$, there exists a $k$-regular $C$-uniform hypergraph $H'\supseteq H$ with the same maximum codegree as $H$.
\end{lemma}

\noindent
{\bf Proof:}\,
A hypergraph is called {\em simple} if its maximum codegree is 1. It is easy to construct a (finite) $k$-regular $(C-1)$-uniform simple hypergraph $H^*=(V^*,E^*)$. For example, let $V^*=U^k$, where $U$ is a set of size $C-1$, and let $E^*$ consist of all $(C-1)$-tuples with $k-1$ fixed coordinates. Let $L$ denote the number of edges in $H^*$.

Let $H=(V,E)$ and $M=\sum_{v\in V}(k-d_H(v))$. Take $M$ isomorphic copies, $H^*_1,\dots,H^*_M,$ of $H^*$ and $L$ isomorphic copies, $H_1,\dots,H_L,$ of $H$ on pairwise disjoint vertex sets, where  $H_1=H$. Construct the hypergraph $H'=(V',E')$ as follows. Let $V'$ be the disjoint union of the vertex sets of all $H_i$ and $H^*_j$. Let $E'$ consist of all the edges of the hypergraphs $H_i$ and the {\em extensions} of the edges of $H^*_j$. Here, an edge of $H^*_j$ is extended with a properly chosen vertex from a hypergraph $H_i$ to obtain a size-$C$ edge of $H'$, making $H'$ $C$-uniform, as required.

We extend different edges in the same copy $H^*_j$ of $H^*$ with vertices from different copies of $H$. This guarantees that if a pair of distinct vertices in $V'$ are contained in more than one edge of $H'$, then they belong to the same $H_i$ and their codegrees in $H'$ and in $H_i$ coincide. Hence, the maximum codegrees in $H$ and $H'$ are the same, as required.

For any $j$, we have to choose a vertex from each copy of $H_i$ to extend one of the edges of $H^*_j$. So, for a fixed copy $H_i$, we have to make $M$ such choices. By the definition of $M$, we can make these choices in such a way that every vertex $v$ in $H_i$ is chosen exactly $k-d_{H_i}(v)$ times. Thus, the degree of $v$ in $H'$ will be precisely $k$. The $H'$-degree of every vertex of $H_j^*$ is also equal to $k$, so $H'$ is $k$-regular.

It follows from the assumption $H_1=H$ that the edge set of $H$ is contained in the edge set of $H'$. This completes the proof of the lemma. \hfill $\Box$
\vspace{0.2cm}

\noindent
{\bf Proof of Theorem \ref{t14}:}\,
Let $N$ be a random necklace with exactly $k$ beads of type $i$ for
each $1 \leq i \leq t$, where $t$ and $k/\log t$ are large.
The total number of
beads is, therefore, $n=kt$.  Fix a small $\eps>0$ and let
$C$ be a large integer.
Let $\delta=\delta(C,\eps)$ satisfy the assertion of Theorem
\ref{t41}. We assume without loss of generality that $\delta<\eps/C$.
Split the necklace into
disjoint intervals, each of length $C$ (with possibly one shorter
interval if $kt$ is not divisible by $C$). This requires fewer than
$n/C$ cuts.
Call an interval {\em bad} if it contains two beads of the same type (or if it is the last interval with fewer than $C$ beads); otherwise call it {\em good}.

Let $H$ be the $C$-uniform hypergraph with $t$ vertices
that represent the $t$ types, and one edge for each good interval consisting of the vertices representing the types of beads in the interval. (Occasionally we might get two good intervals containing beads of the same types, but this represents no problem as we allowed for a multiset for edges.)
The maximum degree in $H$
at most $k$. The proof proceeds by showing that, with high
probability, the maximum codegree in $H$ is at most $\delta k$ and,  thus, Theorem~\ref{t41} applies. One can then take the $k$ largest matchings
and use them to partition almost all beads into $k$ collections without any further cuts,
each containing at most one bead of each type. The remaining few beads
can then be cut loose by a few additional cuts and then they can be placed in the partition classes as required. Since the number of these remaining beads is small, the
total number of cuts is $o(kt)$ as $t$ and $k/\log t$ tend to infinity.
\smallskip

In what follows, we work out the technical details.

Suppose the following two inequalities hold:
\begin{equation}
\label{e41}
t \geq \frac{2C}{\delta},
\end{equation}
\begin{equation}
\label{e42}
t^2e^{-\delta^2k/2} \leq \eps.
\end{equation}
Note that both of these inequalities hold provided that $t$
and $k/ \log t$ are sufficiently large as functions of $C$ and
$\eps$ (which determine $\delta$).

For every pair of distinct types,
$i,j \in [t]$, let $E_{i,j}$ be the event that more than
$\delta k$ edges contain both $i$ and $j$. The probability
of this event can be estimated as
follows. There are at most $k$ intervals containing a bead of type
$i$. When placing the $k$ beads of type $j$ one by one, the
conditional probability for each of them to lie in an interval
containing a bead of type $i$, given any history, is at most
$\frac{(C-1)k}{(t-2)k} <\frac{C}{t}$.
The probability
that there are at least $\delta k$  such beads is, thus, at most the
probability that a binomial random variable with parameters
$C/t$ and $k$ is at least $\delta k$. This can be estimated using (\ref{e41}) and the Chernoff bound (c.f. \cite{AS}, Theorem A.1.4) as
$\mathbb P(E_{i,j})<e^{-\delta^2k/2}$.

Let $E_i$ be the event that among the length $C$ intervals of the necklace $N$, there are at least $\delta k$ which contain at least two beads of type $i$. The probability of this event can be
estimated using the same argument as above.
Indeed, when placing the beads of type $i$ one by one,
every time the probability that it falls into an interval which already contains
a bead of type $i$, is at most $C/t$. If the event $E_i$
occurs, then this happens at least $\delta k$ times, and the
probability of this event is less than $e^{-\delta^2k/2}$.

By (\ref{e42}) and the above estimates, it follows that,
with probability at least $1-\eps$, none of the events
$E_{i,j}$ and $E_i$ holds. Let us assume that this is the case. As none of the events $E_{i,j}$ hold, the maximal codegree in $H$ is below $\delta k$. Therefore, Theorem~\ref{t41} applies and the edges of $H$ can be partitioned into at most $(1+\eps)k$ matchings. We choose the largest $k$ among these matchings (breaking ties arbitrarily) and partition the corresponding good intervals of the necklace into $k$ parts, each containing at most one bead of each
type. We cut each remaining interval (the bad ones and the good ones that do not belong to any of the largest $k$ matchings) into individual beads and distribute these beads appropriately to obtain a fair partition, where each part contains exactly one bead of each type. We used fewer than $n/C$ cuts to obtain the intervals and fewer than $\eps n$ cuts to cut up the good intervals outside the $k$ largest matchings. Finally, we used at most $C-1$ cuts for each bad interval. As none of the events $E_i$ happens, there are fewer than $\delta n+1$ bad intervals. All in all, the number of cuts was smaller than $n/C+\eps n+\delta Cn\le n/C+2\eps n$ (where we used the assumption that $\delta\le\eps/C$). As $\eps>0$
can be chosen arbitrarily small, $C$ arbitrarily
large, and the required conditions hold with probability
at least $1-\eps$ provided that $t$ and $k/\log t$ are sufficiently large
depending on $C$ and $\eps$, the theorem is true. \hfill $\Box$
\vspace{0.2cm}

\noindent
{\bf Proof of Theorem \ref{t15}:}\,
The lower bound  for $X=X(2,t,1)$ is proved by a simple first
moment argument. Let $s$ be the number of cuts allowed, and let
$Y$ be the random variable counting the number of fair
balanced partitions of the necklace using $s (\leq t)$ cuts.
For each fixed
balanced partition, the probability that it is fair is
$$
(t!)^2 \cdot 2^t/(2t)!=\Theta(\frac{\sqrt t}{2^t}).
$$
The number of 2-partitions requiring exactly $s$ cuts is exactly
${{2t-1}  \choose s}$. The number of balanced partitions that can be obtained by at most $s$ cuts is therefore at most $\sum_{i=0}^s{2t-1\choose i} \leq 2^{2H(s/(2t))t}$, where
$H(x)$ is the binary entropy function. Therefore, the expected
number of fair partitions with $s$ cuts is at most
$$
2^{2H(s/(2t))t} \Theta\left(\frac{\sqrt t}{2^t}\right).
$$
For any positive $\eps$, large $t>t_0(\eps)$, and
$s$ smaller than $(2H^{-1}(1/2)-\eps)t$, the last expression is smaller than $\eps$.
This implies that whp $X(2,t,1)$ is at least
$H^{-1}(1/2) 2t -o(t) =0.22...t -o(t)$.
\smallskip

The proof of the upper bound,
obtained jointly with Ryan Alweiss, Colin Defant
and Noah Kravitz, follows.  Let $f(t)$ denote the expectation
of $X(2,t,1)$.
\begin{claim}
\label{c2t1}
$$
f(t+1) \leq f(t) + \frac{1}{2} -
\frac{1}{2} \frac{f(t)}{2t+1}+\frac{1}{2t+2}.
$$
\end{claim}

\noindent
{\bf Proof:}\,
Expose the random necklace with $t+1$ types and $2$
beads of each type
as follows: first expose the
last bead, without loss of generality let its type be $t+1$.
Then expose the necklace of the
$2t$ beads of
types $1, 2,\ldots ,t$ (ignoring the second bead of type $t+1$).
This is uniform random
hence in expectation there is a collection of $f(t)$ cuts of this necklace
of $t$ types that gives a fair partition.
Fix one such minimum collection of cuts.
Let their number be $f$, note that $f$ is a random variable
whose expectation is $f(t)$.
So far we have the relative order of $2t+1$ beads. Now
expose the other occurrence of the bead of type $t+1$, which we call
here the
extra bead: it clearly lies
in a uniform random place among $2t+1$ options which are the
$2t+1$ spaces before bead number $i$ for some $i \leq 2t+1$.
If this place happens to be one of the existing
$f$ cuts then no additional cut is needed, as we can append this
extra bead to an interval of each of the two thieves, as needed, without
adding any new cut. This happens with probability $f/(2t+1)$. If this
is not the case, then with probability roughly a half (computed
precisely in what follows)
the extra bead is placed in
an interval that goes to the thief who is not the one to get the last
interval. Note that in this case we do not need an extra cut either.
The probability that this occurs is
exactly $(t-f/2)/(2t+1-f)$ for even $f$ and
$(t-(f-1)/2)/(2t+1-f)$ for odd $f$. This is because the total number
of beads that this thief gets from among the $2t+1$ beads above is $t$, and
among the spaces just before them, there are $f/2$ which are spaces
among the $f$ cuts if $f$ is even, and $(f-1)/2$ if $f$ is odd.
(In the odd case, each thief gets $(f+1)/2$ intervals,
but the space before the first interval is not a cut). The above
ratio is, thus, exactly $1/2$ for odd $f$ and for even $f$ it satisfies
$$
\frac{1}{2} \frac{2t-f}{2t+1-f} = \frac{1}{2}(1- \frac{1}{2t+1-f})
\geq \frac{1}{2} (1-\frac{1}{t+1}).
$$
Here, we used the fact that $f \leq t$ always holds, by the deterministic
result.
Therefore, the probability of the event above is at least
$\frac{1}{2}-\frac{1}{2t+2}$.
In the only remaining case, which happens with probability at most
$$
(\frac{1}{2}+\frac{1}{2t+2})(1-\frac{f(t)}{2t+1})
\leq \frac{1}{2} - \frac{1}{2} \frac{f(t)}{2t+1}+\frac{1}{2t+2},
$$
we need at most one additional cut: just
before the very last bead (of type $t+1$).  This proves the claim.
\hfill $\Box$
\smallskip

We next show that, by Claim~\ref{c2t1}, $\lim \sup f(t)/t \leq 0.4$.
Indeed, if for some value of $t$, $\frac{f(t+1)}{t+1}
\geq \frac{f(t)}{t}$, then by the claim
$$
\frac{f(t)}{t} \leq \frac{f(t+1)}{t+1} \leq
\frac{f(t)}{t+1} -\frac{1}{2}\frac{f(t)}{(2t+1)(t+1)} +
\frac{1}{2(t+1)}+\frac{1}{2(t+1)^2}.
$$
This gives
$$
\frac{f(t)}{t} \frac{5t+2}{2t+1} \leq 1+\frac{1}{t+1},
$$
implying that in this case $\frac{f(t)}{t} \leq 0.4+O(1/t).$
However, in this case, by the above inequality
(or by the simple fact that $f(t+1) \leq f(t)+1$), we also have
$f(t+1)/(t+1) \leq 0.4 +O(1/t).$

Hence, if there are infinitely many values of $t$
for which $f(t+1)/(t+1) > f(t)/t$ holds, then $\lim \sup f(t)/t \leq 0.4$.
Otherwise, the function $f(t)/t$ is eventually decreasing, so $x=\lim f(t)/t$ exists. In this case Claim~\ref{c2t1} implies
$f(t+1)\le f(t)+1/2-x/4+o(1)$. Summing this for all values $t<t_0$ we get $f(t_0)\le(1/2-x/4)t_0+o(t_0)$ or, equivalently,
$f(t_0)/t_0\le 1/2+x/4 +o(1)$. Taking limits, we obtain $x\le1/2-x/4$, yielding $x\le0.4$ as needed.

We have thus shown that $E(X(2,t,1))\le 0.4t+o(t)$.
We can conclude, by the Azuma-Hoeffding Inequality
(see, e.g., \cite{AS}), that $X(2,t,1)\le 0.4t+o(t)$ holds whp.
Indeed,  by this inequality and the fact that the minimum
number of cuts for a fair partition can change by at most $O(1)$ when
swapping two beads, it follows that the probability that $X(2,t,1)$
deviates from its expectation by at least $C\sqrt t$ is at most
$e^{-\Omega(C^2)}$.  This implies that
$X(2,t,1)\le 0.4t+o(t)$ whp, completing
the proof of Theorem
\ref{t15}.  \hfill $\Box$
\vspace{0.2cm}

\noindent
{\bf Remark.}\, It is possible to slightly improve both estimates
in Theorem~\ref{t15}. The upper bound can be improved by
observing that in the argument above, if the extra bead appears
in an interval forcing us to add a cut, and it happens to appear
between two beads of types $i$ and $j$ where there is a cut just
before the type $i$ bead, then, if the second bead of type $i$ also
appears right after (or before) a cut, it is possible to shift
these two cuts and get a fair partition without increasing the total
number of cuts. There are several similar local scenarios that
can be used in a similar way and lead to small improvements in the
upper bound. Optimizing these arguments can be difficult, but they
do show that the expectation is smaller than $0.39 t +o(t)$ whp.

The lower bound can also be slightly improved as follows.
If one has a fair partition with at most $s$ cuts, then one can
choose that to be minimal: first, minimize the number of cuts, then
try to make the cuts as far to the right as possible.
Now the first beads after the cuts have all distinct types, as
otherwise two of the cuts could be shifted one position to the right (and if
they reach the next cut, then they would cancel). This gives a tiny
advantage in the probability (probability of a set of cuts being
fair
versus probability of being the minimal fair set). This gives a
lower
bound of roughly $0.227 \cdot t$,
a tiny improvement over $0.22 \cdot t$.
These considerations still leave a substantial gap between the
upper and lower bounds.

\section{Concluding remarks and open problems}\label{sectionConcl}

We have studied the minimum possible number of cuts required to
partition a random necklace with $km$ beads of each of $t$ types
fairly into $k$ collections. This
minimum is denoted by $X=X(k,t,m)$.
A better understanding of the behavior of the random variable $X$
for all admissible values of the parameters requires further study.
\begin{itemize}
\item
Some of the arguments described here
for $k=2$, fixed $t$ and large $m$ can be extended to higher
values of $k$. In particular, we note that the argument in the
proof of the upper
bound for the probability that $X(2,3,m)=O(1/\log m)$ implies that
for every fixed $k >2$ and $t=3$, the probability that
$X(k,3,m)=\frac{(k-1)(t+1)}{2}=2k-2$ is also $O(1/\log m)$.
Indeed, $2k-2$ cuts split the necklace into $2k-1$ intervals,
hence there is at least one thief who gets a single interval. This
interval has to contain exactly $m$ beads of each of the three
types. The probability of the existence of such an interval is
$O(1/\log m)$, by the argument in the proof presented in Subsection~2.1.
\item
We have seen that the
ratio $X(2,t,1)/t$ is between $0.22$ and $0.4+o(1)$ whp. Is the
ratio $c+o(1)$ whp, for some constant $c$, and if so,
what is the value of $c$?
\item
The algorithmic problem of finding a small number of cuts
yielding a fair partition efficiently for a given input random
necklace is also interesting. For the deterministic case there are
known hardness results for the problem (see \cite{FG}) and known
approximation algorithms (\cite{AG}), and it will be interesting to find
efficient algorithms that  work better whp for the random case.
\item
The random variable $X(2,t,m)$ has the following interpretation in
terms of a question about folding positive random walks in $Z^t$.
Consider a random walk of $2tm$ steps in $Z^t$. Starting from the
origin, every step is one of the $t$ unit vectors $e_i$, where the
sequence of steps is a random sequence consisting of exactly $2m$
steps in each direction $e_i$. An elementary folding at $j$ of this
sequence of steps switches the signs of all steps from step number
$j$ until the end. We can apply several of these elementary foldings one after the other. The random variable
$X(2,t,m)$ is thus the minimum number of elementary foldings required to
ensure the folded sequence ends at the origin. A similar question,
avoiding parity issues,
can be considered for a sequence of $n$ random, independent,
positive steps in $Z^t$. The random variable now is the minimum
number of elementary foldings required to ensure the folded walk
ends within $\ell_{\infty}$-distance $1$ from the origin.
\end{itemize}

\noindent
{\bf Acknowledgment}.
We thank Yuval Peres for providing helpful references, and thank
Ryan Alweiss, Colin Defant and Noah Kravitz for their help in the
proof of the upper bound in Theorem \ref{t15}.


\begin{thebibliography}{99}

\bibitem{Al}
N. Alon, Splitting necklaces, Advances in Mathematics 63 (1987),
247--253.

\bibitem{AG}
N. Alon and A. Graur,
Efficient splitting of necklaces,
Proc. ICALP 2021, Article No. 14, pp. 14:1--14:17.
See also: arXiv:2006.16613.

\bibitem{AS}
N. Alon and J. H. Spencer, The Probabilistic Method,
Fourth edition,
Wiley, 2016, xiv+375 pp.

\bibitem{CE}
K. L. Chung and P. Erd\H{o}s,
On the application of the
Borel-Cantelli lemma, Transactions of the American
Mathematical Society 72 (1) (1952), 179--186.

\bibitem{DE}
A. Dvoretzky and P. Erd\H{o}s,
Some problems on random walk in
space, In Proc. 2nd Berkeley Symp, 1951, pp. 353--367.

\bibitem{erdos taylor2}
P. Erd\H{o}s and S. J. Taylor, Some intersection
properties of random walk paths, Acta Math. Acad. Sci.
Hungar. 11 (1960), 231--248.

\bibitem{ET}
P. Erd\H{o}s and S. J. Taylor,
Some problems concerning
the structure of random walk paths,
Acta Mathematica Academiae
Scientiarum Hungarica 11, no. 1-2 (1963), 137--162.

\bibitem{FG}
%Aris Filos-Ratsikas and Paul~W. Goldberg.
A. Filos-Ratsikas and P. W. Goldberg,
The Complexity of Splitting Necklaces and Bisecting Ham
Sandwiches.
Proceedings of the 51st Annual ACM Symposium on Theory of
Computing (STOC), pages 638--649, 2019.

\bibitem{Hoe}
W. Hoeffding, Probability inequalities for sums of bounded random
variables, J. Amer. Stat. Assoc. 58 (1963), 13--30.

\bibitem{JLR}
S. Janson, \L uczak and A. Ruci\'nski,
Random graphs,
Wiley-Interscience Series in Discrete Mathematics and Optimization.
Wiley-Interscience, New York, 2000. xii+333 pp.

\bibitem{Lawler1}
G. F. Lawler, The probability of intersection of independent
random walks in four dimensions, Communications in Mathematical
Physics 86 (4) (1982), 539--554.

\bibitem{Lawler2}
G. F. Lawler,  Intersections of random walks in
four dimensions. II, Communications in
mathematical physics 97 (4) (1985), 583--594.

\bibitem{PS} N. Pippenger and J. Spencer,
Asymptotic behavior of the chromatic index for hypergraphs,
J. Combin. Theory Ser. A 51 (1989), 24--42.

\bibitem{PZ} R. E. A. C. Paley and A. Zygmund, On some
series of functions, (3), Mathematical Proceedings of the
Cambridge Philosophical Society. 28 (2) (1932), 190--205.

\bibitem{PZ1} R. E. A. C. Paley and A. Zygmund,
A note on analytic functions in the unit circle, Mathematical
Proceedings of the Cambridge Philosophical Society. 28 (3) (1932),
266--272.

\bibitem{polya}
G. P\'olya, \"Uber eine Aufgabe der Wahrscheinlichkeitsrechnung
betreffend die Irrfahrt im Stra\ss ennetz,
Mathematische Annalen 84 (1) (1921), 149--160.

\bibitem{Berry esseen}
M. Rai\v{c}, A multivariate Berry-Esseen theorem with explicit constants,
Bernoulli 25 (4A) (2019), 2824--2853.

\bibitem{Re} P. R\'ev\'esz, Random Walks in Random and Non-Random
Environments,
Third edition,
World Scientific Publishing Co. Pte. Ltd., Hackensack, NJ, 2013.
xviii+402 pp.

\bibitem{LLT}
B. Trojan, Long time behavior of random walks on
the integer lattice, Monatshefte f\"ur Mathematik 191 (2) (2020), 349--376.

\end{thebibliography}
\end{document}